\documentclass{amsart}
\usepackage{cjhebrew}
\usepackage{url}

\usepackage[backref=page,colorlinks,bookmarksopen=true]{hyperref}
\usepackage{amsfonts,amsmath,amssymb}
\usepackage{amsthm}
\usepackage{enumitem}
\usepackage{appendix}
\usepackage{pgf,tikz}
\usetikzlibrary{arrows}

\usepackage{mathrsfs, xspace}
\usepackage{stmaryrd}
\usepackage{graphicx,psfrag}
\usepackage{a4wide}
\usepackage[all]{xy}
\theoremstyle{plain}
\newtheorem{proposition}{Proposition}[section]
\newtheorem{theorem}[proposition]{Theorem}

\newtheorem{lemma}[proposition]{Lemma}
\newtheorem{corollary}[proposition]{Corollary}

\theoremstyle{definition}
\newtheorem{definition}[proposition]{Definition}

\theoremstyle{remark}
\newtheorem{remark}[proposition]{Remark}

\everymath{\displaystyle}

\DeclareMathOperator{\diam}{diam}
\DeclareMathOperator{\rad}{rad}

\DeclareMathOperator{\Conv}{Conv}

\DeclareMathOperator{\Isom}{Isom}

\newcommand{\bd}{\partial}

\newcommand{\action}{\curvearrowright}

\newcommand{\R}{\mathbb{R}}                          % |R
\newcommand{\N}{\mathbb{N}}                          % |N
\newcommand{\X}{\mathbf{X}}
\newcommand{\bN}{\mathbb{N}}  
\DeclareMathOperator{\dd}{d\!}

\author{Uri Bader, Bruno Duchesne and Jean L\'ecureux}
\date{\today}
\begin{document}
\title{Furstenberg maps for CAT(0) targets of finite telescopic dimension}
\thanks{We warmly thank Alex Furman who kindly permit us to include the material of \S \ref{sec:metricErg}.}
\maketitle

\begin{abstract}We consider actions of locally compact groups $G$ on certain CAT(0) spaces $X$ by isometries. The CAT(0) spaces we consider have finite dimension at large scale. In case $B$ is a $G$-boundary, that is a measurable $G$-space with amenability and ergodicity properties, we prove the existence of equivariant maps from $B$ to the visual boundary $\partial X$.
\end{abstract}
\section{Introduction}

Furstenberg maps, or boundary maps, first appeared in H. Furstenberg's work \cite{MR0146298,MR0352328}. These maps proved to be powerful tools for rigidity results.  Indeed, the existence of such Furstenberg maps  is used in particular in order to prove commensurator rigidity \cite{MR1253187} or superrigidity phenomena \cite{MR1090825}.

Our main topic of investigation in this paper is the existence of Furstenberg maps in the context of actions of groups on CAT(0) spaces. Recall that a CAT(0) space is a complete metric space that is non-positively curved in a way defined via the Bruhat-Tits inequality \cite[p.~163]{MR1744486}. 

We do not restrict ourselves to locally compact CAT(0) spaces but we replace this hypothesis by a condition at large scale. A CAT(0) space has finite telescopic dimension if any of its asymptotic cones has finite geometric dimension. Recall that a CAT(0) space has finite geometric dimension if there is a finite upper bound on the topological dimensions of its compact subspaces. Geometric dimension was introduced by B. Kleiner \cite{MR1704987} and telescopic dimension by P.-E. Caprace and A. Lytchack \cite{MR2558883}. 

For example, CAT(0) cube complexes with an upper bound on the dimensions of cubes, Euclidean buildings --- not necessarily locally compact nor with discrete affine Weyl groups ---  and infinite dimensional symmetric spaces with non-positive operator curvature and finite rank \cite{Duchesne:2012xy} have finite telescopic dimension.

Our main result is the proof of existence of Furstenberg maps for such spaces.

\begin{theorem}\label{main} Let $X$ be a CAT(0) space of finite telescopic dimension and let $G$ be a locally compact second countable group acting continuously by isometries on $X$ without invariant flats. If $(B,\nu)$ is a $G$-boundary  then there exists a measurable $G$-map from $B$ to $\bd X$.
\end{theorem}

Roughly speaking a $G$-boundary is a measurable space with an amenable $G$-action and strong ergodic properties (see \S \ref{sec:metricErg} for a precise definition). In case $G$ is amenable, $B$ can be chosen to be a point and our theorem reduces to the following statement \cite{MR2558883}: If $G$ acts continuously by isometries  on a CAT(0) space of finite telescopic dimension without invariant flat subspace then there is a fixed point at infinity, that is there is a (trivial) map from $B$ to $\bd X$. Our tools to pass from that last statement to Theorem \ref{main} are \emph{measurable fields of CAT(0) spaces} over $B$, see \S \ref{measurable}.

In this theorem, the $G$-boundary can be taken to be any Poisson boundary of $G$ (for an admissible measure). However, the theorem can be proved for an {\it a priori} larger class of spaces $B$, with suitable ergodic properties. Our work relies heavily on a very strong ergodic property of the boundary, namely \emph{relative metric ergodicity}. This property was introduced by the first author together with A.~Furman in \cite{Bader:2013zr}, where (generalizing results of \cite{MR1911660} and \cite{MR2006560}) it is proved that the Poisson boundary has indeed this strong ergodic property. For the sake of completeness, we also include a proof in this paper. 

We also prove a similar result for \emph{boundary pairs}, see Theorem \ref{mainpairs}. For example, Poisson boundaries associated to forward and backward random walks form a boundary pair.\\
%
%
%\begin{theorem}\label{boundary}Let $G$ be a locally compact second countable group endowed with a symmetric probability measure $\mu$, then  the Poisson boundary $(B,\nu)$ of $G$ associated to $\mu$ is a $G$-boundary.
%\end{theorem}
%Our main motivation in starting this work was to generalize superrigidity results for infinite-dimensional symmetric spaces, or exotic affine buildings. Our work in this direction will appear in a subsequent paper. 

Another possible approach to prove Theorem \ref{main}, at least for the Poisson boundary associated to some random walk on $G$, would be to try to understand the behavior of the random walk itself. More precisely, if $Z_n$ is the $n$-th step of the random walk in $G$, and $o\in X$, one could hope that $Z_n.o$ converges to a point in the visual boundary. This would give a measurable, $G$-equivariant map from $B$ to $\bd X$. 

It is known that this approach works in some cases by the work of A.~Karlsson and G.~Margulis
 \cite[Theorem 2.1]{MR1729880}. However, there is a crucial assumption in this theorem, namely, that the drift is positive.
Our result doesn't rely on this hypothesss. 
 
 In general, it is a natural question to relate the map obtained by Theorem \ref{main} and the random walk in the space $X$. For example, if the drift is positive, and with finite first moment, we know that the random walk converges to a boundary point, and one can wonder whether the boundary map can be far away (in the sense of the Tits metric) from this limit. In case $X$ is Gromov hyperbolic, the two essentially coincide. This is not the case for a symmetric space for example, as there are several different boundary maps.

In the same vein, one can ask whether the boundary map can be a (measurable) isomorphism, thus providing a geometric identification of the Poisson-Furstenberg boundary. The main tools for this possible application would be 
V.~Kaimanovich's ray and strip criteria \cite[Theorems 5.5 \& 6.4]{MR1815698}.\\

Finally, it also turns out that the methods on which our paper relies are also useful to answer a natural question raised by an implicit argument in \cite{MR2558883}. Namely, in  Appendix \ref{nonemptyness}, we prove the following result of non-emptiness for boundaries of CAT(0) spaces of finite telescopic dimension. Quite surprisingly, its proof relies on the existence of $(G,\mu)$-boundaries for countable groups $G$ and the dichotomy given between existence of Furstenberg maps and existence of Euclidean subfields given by Theorem \ref{AB}.

\begin{theorem}\label{nem}Let $X$ be a CAT(0) space of finite telescopic dimension. If  $\Isom(X)$ has no fixed point then the visual boundary $\bd X$ is not empty.
\end{theorem}

A CAT(0) space $X$ is called \emph{boundary minimal} \cite[\S 1.B]{MR2574740} if there is no non-empty closed convex subspace $Y\subset X$ such that $\bd Y=\bd X$. For proper CAT(0) spaces with boundary of finite dimension, minimality of $X$ (that is no non-trivial closed convex subset is invariant under $\Isom(X)$) implies boundary minimality  \cite[Proposition 1.5]{MR2574740}. The same holds for CAT(0) spaces of finite telescopic dimension.

\begin{corollary}\label{bm}Let $X$ be a minimal CAT(0) space of finite telescopic dimension not reduced to a point. Then $X$ is boundary minimal.
\end{corollary}
%
%The organization of this paper is as follows. We start in section \ref{sec:metricErg} by reviewing the different notions of ergodicity, metric ergodicity and relative metric ergodicity, and prove that the Poisson-Furstenberg boundary satisfies these properties. Then in section \ref{Bmap} we explain the main tools from CAT(0) geometry that we need, and proceed to the proof of Theorem \ref{main}. In order to lighten the text, we gathered the definitions and results that we need about measurable fields of metric spaces in Appendix A. Finally, appendix B is devoted to the proof of Theorem \ref{nem}.

\
\section{CAT(0) geometry}\label{CAT(0)}

\subsection{CAT(0) spaces of finite telescopic dimension} All along this text, we will deal with CAT(0) spaces \emph{of finite telescopic dimension} as introduced in \cite{MR2558883}. For reader's convenience, we recall facts about these spaces which will be useful for us.

We will start by defining dimension in a metric way, by
 Jung's Theorem \cite{Jung1901}. It is shown in the aforementioned paper that for any bounded subspace $Y$ of a Euclidean space of dimension $n$

\begin{equation}\label{Jung}
\rad(Y)\leq\sqrt{\frac{n}{2(n+1)}}\diam(Y)
\end{equation}

and there is equality if and only if the closure of $Y$ contains a regular $n$-simplex of diameter $\diam(Y)$.

Using this inequality, P.-E. Caprace and A. Lytchack showed \cite[Theorem 1.3]{MR2558883} that a CAT(0) space $Y$ has \emph{geometric dimension} (as introduced by B. Kleiner \cite{MR1704987}) at most $n$ if and only if for any bounded subset $Y\subseteq X$, Inequality \eqref{Jung} holds.

\emph{Telescopic dimension} is a notion at large scale. More precisely, a CAT(0) space $X$ has telescopic dimension at most $n$ if and only if any asymptotic cone of $X$ has geometric dimension at most $n$. It can be expressed quantitatively: a CAT(0) space $X$ has telescopic dimension at most $n$ if and only if for any $\delta>0$, there is $D>0$ such that for any bounded subset $Y\subseteq X$ of diameter larger than $D$ we have 

\begin{equation} 
\rad(Y)\leq\left(\delta+\sqrt{\frac{n}{2(n+1)}}\right)\diam(Y).
\end{equation}
Note that a locally compact CAT(0) space may have infinite telescopic dimension.

One main feature of CAT(0) spaces of finite telescopic dimension is the following: if the intersection of a filtering family $\{X_\alpha\}$ of closed convex subsets is empty then the intersection of  boundaries $\cap_\alpha\bd X_\alpha$ is not empty and there is a canonical point $\xi$ in this intersection.

This canonical point is given by the fact that the boundary of a CAT(0) space of telescopic dimension at most $n$ has geometric dimension at most $n-1$ and the fact that a CAT(1) space $\Xi$ of finite geometric dimension and radius at most $\pi/2$ has such a canonical point. This point is  defined as the unique circumcenter of the set of circumcenters of $\Xi$.

This property of filtering families of closed convex subspaces can be thought as a compactness property for $\overline{X}=X\cup\bd X$. Let us define the topology $\mathscr{T}_c$ on $\overline{X}$ as the weakest topology such that $\overline{C}$ is $\mathscr{T}_c$-closed for any closed --- for the usual topology --- convex subset $C\subseteq X$ \cite[\S 3.7\&\S 3.8]{MR2219304}. Then for a CAT(0) space $X$ of finite telescopic dimension $\overline{X}$ is compact --- not necessarily Hausdorff. 

\subsection{Geometry of flats in CAT(0) spaces} In the heart of the proof of Theorem \ref{main} we will deal with a Euclidean subfield of a CAT(0) field. We gather in this subsection useful facts about the geometry of Euclidean subspaces  (also called flats subspaces) in CAT(0) spaces.\\

Let $(X,d)$ be a CAT(0) space. 
 If $C$ is a closed convex subset of $X$, we denote by $d_C$ the distance function to $C$, that is $d_C(x)=\inf_{c\in C}d(c,x)=d(x,\pi_C(x))$ where $\pi_C(x)$ is the projection of $x$ on $C$.
 
 Let $Y$ be a bounded subset of a metric space $(Z,d)$. The \emph{circumradius} (or simply \emph{radius}) $\rad(Y)$ of $Y$ is the non-negative number $\inf_{z\in Z}\sup_{y\in Y}d(y,z)$, its \emph{intrinsic circumradius} is $\inf_{z\in Y}\sup_{y\in Y}d(y,z)$ and a \emph{circumcenter} of $Y$ is a point in $Z$ minimizing $\sup_{y\in 
Y}d(y,\cdot)$.

More generally, we define a \emph{center} of $Y$ as a point of $Z$ which is fixed by any isometry of $Z$ stabilizing $Y$. The Bruhat-Tits fixed point lemma asserts that, in a CAT(0) space, a bounded set has a unique circumcenter, which is therefore a center.\\

If $S_1$ and $S_2$ are two subsets of Euclidean spheres, we denote by $S_1\ast S_2$ their spherical join \cite[Definition I.5.13]{MR1744486}. This is the spherical analogue of Euclidean products. Such Euclidean products appear, for example, in the de Rham decomposition \cite[Theorem II.6.15]{MR1744486} of the CAT(0) space $X$: the space $X$ is isometric to a product $H\times Y$ where $H$ is a Hilbert space and $Y$ that does not split with a Euclidean factor. Isometries of $X$ preserve this decomposition acting diagonally on $H\times Y$ and the choice of a base point in $X$ allows us to identify $H$ and $Y$ with closed convex subspaces of $X$ containing this point.\\
%
%
%
%\begin{lemma}
%Let $X$ be a CAT(0) space, and $C\subset \bd X$ be a nonempty subset of intrinsinc radius $\leq \pi/2$. Then $C$ has a center.
%\end{lemma}

The following lemma details the possibilities for a convex function on a Euclidean space. It will be applied in two situations: the restriction of a Busemann function to a Euclidean subspace of $X$ and to the restriction to a Euclidean subspace of the distance function to another Euclidean subspace.

\begin{proposition}\label{relpos}Let $E$ be a Euclidean space, $f$ be a convex function on $E$ and $m=\inf \{f(x)\mid x\in E\}$. Then :

\begin{enumerate}[label=(\roman*)]
\item If $m$ is not attained, then the intersection $\displaystyle\bigcap_{\varepsilon>0}\bd\left (f^{-1}\left((m,m+\varepsilon)\right)\right)$ is not empty and has a  center.
\end{enumerate}

If $m$ is a minimum, let $E_m=f^{-1}(\{m\})$. Let $E_m=F_m\times T$ be its de Rham decomposition, where $F_m$ is a maximal flat contained in $E_m$. Then exactly one of the following possibilities happen. 

\begin{enumerate}[resume,label=(\roman*)]
\item  $E_m$ is bounded and thus has a center (i.e. $F$ is a point and $T$ is compact),
\item $T$ is bounded and $\bd E_m=\bd F_m$ is a sphere,
\item $T$ is unbounded and its  boundary $\bd T$ has radius less than $\pi/2$ and thus has a center.
\end{enumerate}
\end{proposition}

\begin{proof} If $m$ is not a minimum then the net of closed convex subsets $\left (f^{-1}\left((m,m+\varepsilon)\right)\right)_{\varepsilon>0}$ has empty intersection and the result follows from $\mathscr T_c$-compactness.

The other cases coincide with the join decomposition $\bd E_m=\bd F*\bd T_m$, using Lemma \ref{sph} below where $\bd E_m$ will correspond to $C$, $\bd F_m$ to $S_1$ and $\bd T$ to $C_2$.
\end{proof}

\begin{lemma}\label{sph}Let $(S,d)$ be a Euclidean sphere. Let $C$ be a non-empty closed convex subset of $S$ and let $S_0$ be the minimal subsphere of $S$ containing $C$. Then there is a unique decomposition of $S_0$ as a spherical join $S_0=S_1*S_2$ where $S_1$ and $S_2$ are  reduced to a point or are subspheres of $S_0$ such that
\[C=S_1*C_2\]
where $C_2$ is a closed convex subset of $S_2$ with intrinsic radius $<\frac{\pi}{2}$.

In particular, any closed convex subset that is not a subsphere has a center (the one of $C_2$). Moreover,  it coincides with the unique circumcenter of the set of circumcenters of $C$.\end{lemma}

\begin{proof}First observe that intersections of subspheres are empty or subspheres themselves. This yields the existence of $S_0$. In the same way, convex hulls of subspheres are subspheres and thus there exists a maximal subsphere $S_1$ contained in $C$. Let $S_2$ be the set of points of $S_0$ at distance $\pi/2$ from every points in $S_1$. Then $S_0=S_1*S_2$. Any point of $C$ can be written $(x_1,x_2,\alpha)$ with $x_i\in S_i$ and $\alpha\in[0,\pi/2]$. Since $C$ is convex and $S_1\subseteq C$, there exists $C_2$, which is $S_2\cap C$, such that $C=S_1*C_2$. Observe that $C_2$ does not contain any sphere by maximality of $S_1$. Now $C$ has diameter $<\pi$ --- otherwise, it contains at least a sphere of dimension 0 (that is two antipodal points) --- thus has also intrinsic circumradius $<\pi/2$.

Observe that $C$ is not a sphere if and only if $C_2$ is not empty. In the case where $S_1$ and $C_2$ are not empty then any point of $C_2$ is a circumcenter since in that case, the intrinsic circumradius is $\pi/2$. The fact that any convex subset of circumradius $<\pi/2$ has a unique circumcenter implies the last sentence of the lemma.\end{proof}

There is a particular situation for the relative position of two Euclidean subspaces $E,F$ in $X$: when the restriction on $F$ of the distance to $E$ is constant and vice-versa. In that situation $E$ and $F$ are said to be \emph{parallel}. The Sandwich Lemma \cite[Exercise II.2.12(2)]{MR1744486} implies that their convex hull splits isometrically as $\R^n\times [0,d]$. In particular, $E$ and $F$ are isometric and thus have same dimension $n$.

\begin{lemma}\label{flatproduct} Let $E$ be a Euclidean subspace of $X$ and let $Y$ be the union of subspaces parallel to $E$. Then $Y$ is a closed convex subspace of $X$. A point $y\in X$ belongs to $Y$ if and only if for any $x_1,\dots,x_n\in E$,  $\Conv(y,x_1,\dots,x_n)$ is isometric to a convex subset of a Euclidean space.

Let $p$ be the restriction to $Y$ of the projection to $E$. Fix some $x\in E$ and let $Z$ be $p^{-1}(\{x\})$. Then $Z$ is a closed convex subspace of $X$ and $Y$ decomposes isometrically as $E\times Z$.
\end{lemma}

\begin{proof}The lemma is proved when $E$ is a line in \cite[Proposition II.2.14]{MR1744486}. Let $n$ be the dimension of $E$. We proceed by induction on $n$. Assume this is true for $n-1$ and choose an orthogonal splitting $F\times L$ of $E$ where $F$ has dimension $n-1$, $L$ is a line and $F\cap L=\{x\}$. The induction assumption for $F$ implies the convex hull of subspaces parallel to $F$ splits isometrically as $F\times Z_F$. Now let us apply the case $n=1$ for the union of lines parallel to $L$ in $Z_0$ which splits as  $L\times Z_0\subseteq Z_F$. It is clear that $E\times Z_0\subset Y$. 

If $E'$ is a Euclidean subspace parallel to $E$ then the restriction $p'$ of $p$ to $E'$ is an isometry. Set $F'=p'^{-1}(F)$ and $L'$ to be the orthogonal line to $F'$ containing $x'=p'^{-1}(x)$. For any $y\in L'$, $F\times\{y\}$ is parallel to $F$ and thus $L'\subset Z_F$. Since $L'$ is parallel to $L$, $x'\in Z_0$. Finally $Z_0=Z$ and $Y$ splits isometrically as $E\times Z$.

Assume that  for any $x_1,\dots,x_n\in E$,  $\Conv(y,x_1,\dots,x_n)$ is isometric to a convex subset of a Euclidean space. Then $\Conv(\{y\}\cup E)$, which is the union of such spaces, satisfies the following property: for any $z_1,\dots,z_n\in \Conv(\{y\}\cup E)$, $\Conv(z_1,\dots,z_n)$ is isometric to a convex subset of a Euclidean space. This property is a characterization of CAT(0) spaces that are isometric to a convex subset of a Hilbert space. In particular, in our case $\Conv(\{y\}\cup E)$ is necessarily isometric to $\R^d\times d(y,E)$ where $d=\dim(E)$.
\end{proof} 

\begin{lemma}\label{flatbusemann} Let $E$ be a Euclidean subspace of $X$ and $\xi\in \bd X$ such that the Busemann function $\beta_\xi$ associated to $\xi$ (with respect to some fixed base point) is constant on $E$. If $x\in E$ and $\rho$ is the ray from $x$ to $\xi$ then the convex hull of $E\cup\rho$ is isometric to $E\times [0,+\infty)$.

In particular, with the notations of Lemma \ref{flatproduct}, $\xi\in \bd Z$.
\end{lemma}

\begin{proof} Choose $x\in E$ and $y$ on the geodesic ray from $x$ to $\xi$. We claim that $\overline\angle_x(y,z)=\pi/2$ for any point $z\neq x$ in $E$. Assume $\overline\angle_x(y,z)<\pi/2$. Then there is $x'\in(x,z)$ such that $d(y,x')<d(y,x)$. This implies $\beta_\xi(x')<\beta_\xi(x)$, which is a contradiction. Arguing the same way with the symmetric of $z$ with respect to $x$, we get the claim. In particular, for any $x,x'\in E$, $\angle_x(x',\xi)+\angle_{x'}(x,\xi)=\pi$ and \cite[Proposition II.9.3]{MR1744486} implies that the convex hull of $x,x'$ and $\xi$ is isometric to $[0,d(x,x')]\times[0,\infty)$. 

This shows that the projection of $E$ on any closed horoball is a flat parallel to $E$ and is the lemma is now a consequence of Lemma  \ref{flatproduct}.
\end{proof}

\section{Measurable fields of complete separable metric spaces}\label{measurable}

\subsection{Metric fields}
An important --- although slightly technical --- tool in our proof will be the notion of fields of metric spaces. Roughly speaking a measurable field of metric spaces over a measurable space $A$ is a way to attach measurably a metric space to any point in $A$. Thanks to \cite{MR0486390} one can think to measurable fields of metric spaces over $A$ in the following way: to each point of $A$ one associates a closed subspace of a fixed  metric space, namely the Urysohn space.
 %We recall here the main definitions and refer to the previous references for more details.

%Let us recall that a complete separable space is a separable topological space whose topology can be defined by a complete metric. Following A. Kechris \cite{MR1321597}, we say that a metric space $(X,d)$ is a complete separable metric space if $d$ is complete and the space is separable. The difference between the two notions is the choice of a particular complete compatible metric.

Let $(A,\eta)$ be a Lebesgue space, that is a standard Borel space with a Borel probability measure \cite[\S 12]{MR1321597}. All our definitions will only depend on the class of the measure $\eta$.

\begin{definition} Let $\{X_a\}_{a\in A}$ be a collection of complete separable metric spaces. The distance on $X_a$ is denoted $d_a$ or simply $d$ if there is no ambiguity. Measurablity conditions are defined thanks to the notion of fundamental families. A \emph{fundamental family} $\mathcal F=\{x^n\}$ is a countable set of elements of 
$\prod_{a\in A} X_a$ %$\bigsqcup_{a\in A} X_{a}\to A$ 
with the following properties
\begin{itemize}
\item $\forall n,m$, $a\mapsto d_a(x^n_a,x^m_a)$ is measurable,
\item for almost every $a\in A$, $\{x^n_a\}$ is dense in $X_a$.
\end{itemize}

A \emph{measurable field $\mathbf{X}$ of complete separable metric spaces} ---or simply a \emph{ metric field} for short--- is then the collection of data: $(A,\eta)$, $\{X_a\}_{a\in A}$ and $\{x^n\}$.\\

A \emph{section} of $\X$ is an element $x\in\prod_a X_a$ such that for all $y\in\mathcal{F}$, $a\mapsto d_a(x_a,y_a)$ is measurable. Two sections are identified if they agree almost everywhere. The set of all sections is the \emph{measurable structure} $\mathcal{M}$ of $\X$. If $x,y$ are two sections, the equality
\[d_a(x_a,y_a)=\sup_{z\in\mathcal{F}}|d_a(x_a,z_a)-d_a(z_a,y_a)|\]
shows that $a\mapsto d_a(x_a,y_a)$ is also measurable.
\end{definition}
Let $G$ be a second countable locally compact group.  The Lebesgue space $(A,\eta)$ is a $G$-\emph{space} if $G$ acts by measure class preserving automorphisms on $A$ and  the map $(g,a)\mapsto ga$ is measurable.

\begin{definition} Let  $(A,\eta)$ be a $G$-space. A \emph{cocycle} for $G$ on $\X$ is a collection $\{\alpha(g,a)\}_{g\in G,a\in A}$ such that 
\begin{itemize}
\item for all $g$ and almost every $a$, $\alpha(g,a)\in$ Isom$(X_a,X_{ga})$;
\item for all $g,g'$ and almost every $a$, $\alpha(gg',a)=\alpha(g,g'a)\alpha(g',a)$;
\item for all $x,y\in\mathcal{F}$, the map $(g,a)\mapsto d_a(x_a,\alpha(g,g^{-1}a)y_{g^{-1}a})$ is measurable.\\
\end{itemize}
In that case we say that $G$ \emph{acts} on $\mathbf{X}$ via the cocycle $\alpha$ or that there is an \emph{action} of $G$ on $\mathbf{X}$. A section $x$ is invariant if for all $g$ and almost all $a$, $\alpha(g,g^{-1}a)x_{g^{-1}a}=x_a$.
\end{definition}

\subsection{Fields of CAT(0) spaces}

A special case of metric fields is the case of CAT(0) fields. The theory of measurable fields of complete separable metric spaces and more specifically of CAT(0) spaces appeared in \cite{MR3163023,MR3044451} and references therein.
\begin{definition}Let $\mathbf{X}$ be a metric field and $\kappa\in\R$. We say that $\mathbf{X}$ is a CAT($\kappa$) \emph{field} if for almost every $a$, $X_a$ is a CAT($\kappa$) space.

A \emph{subfield} $\mathbf{Y}$ of a CAT(0) field $\X$ is a collection $\{Y_a\}_{a\in A}$ of  non-empty closed convex subsets such that for every section $x$ of $\X$, the function $a\mapsto d(x_a,Y_a)$ is measurable. We identify subfields $\mathbf{Y}$ and $\mathbf{Y}'$ if $Y_a=Y_a'$ for almost every $a$.

Similarly we speak of \emph{Euclidean fields} and subfields of such fields. If a group $G$ acts on $\mathbf X$, a subfield $\mathbf Y$ is \emph{invariant} if for all $g$ and almost all $a$, $\alpha(g,g^{-1}a)Y_{g^{-1}a}=Y_a$.
\end{definition}

In CAT(1) spaces, subsets of circumradius less than $\pi/2$ are strictly convex and admits a unique circumcenter as well, see \cite[Proposition II.2.7]{MR1744486} or \cite[Proposition 3.1]{MR1456512} for quantitative statements. Those circumcenters can be defined canonically by means of the metric structure and arguing as in \cite[Lemma 8.7]{MR3044451} we obtain the following.

\begin{lemma}\label{cenCAT1}Let $\mathbf{X}$ be a CAT(1) field. If $\mathbf{C}$ is a subfield of $\mathbf{X}$ with fibers of radius less than $\pi/2$ then the family of circumcenters of $\mathbf{C}$ is a section of $\mathbf{X}$.
\end{lemma}

\begin{lemma}\label{field balls} Let $\mathbf X$ be a metric field and $x$ be a section of it. For $a\in A$, let $B_a^r$ be the closed ball of radius $r$ around $x_a$ then $\mathbf B^r=(B_a^r)$ is a metric field. Moreover if $G$ acts on $\mathbf X$ and $x$ is an invariant section then $G$ acts on $\mathbf B$ as well.
\end{lemma}

\begin{proof}This statement is almost a straightforward verification of the definitions. The only non completely obvious fact is maybe the construction of a fundamental family. Fix a fundamental family $x^n$ of $\mathbf X$. For $n\in\N$, $a\in A$, set inductively $k^n_a$ to be $\min\{k>k^{n-1}_a;\ x^k_a\in \overline B(x_a,r)\}$. Let us denote by $y^n_a$ the point $x^{k^n_a}_a$. The family $(y^n)$ is a fundamental family of $\mathbf B^r$.
\end{proof}
 
Building on Lemma \ref{flatproduct} we obtain the following measurable decomposition of the union of flats parallel to a Euclidean subfield.

\begin{lemma}\label{parallel field}Let $\mathbf X$ be a CAT(0) field and $\mathbf E$ be a Euclidean subfield. For $a\in A$ let $Y_a$ be the union of flats parallel to $E_a$. Then $\mathbf Y=(Y_a)$ is a subfield of $\mathbf X$ which splits as a product of CAT(0) fields $\mathbf Y=\mathbf E\times \mathbf Z$. Moreover if $G$ acts on $\mathbf X$ and $\mathbf E$ is invariant then $\mathbf Y$ is invariant and $G$ acts diagonally on  $\mathbf E\times \mathbf Z$ with an invariant section in $\mathbf Z$.
\end{lemma}

\begin{proof} Fix $a\in A$. Thanks to Lemma  \ref{flatproduct}, the condition $y\in Y_a$ can checked only using distances $d(y,x^n_a)$ where $(x^n)$ is a fundamental family of the subfield $\mathbf E$. One can then readily check that $\mathbf Y$ is a subfield of $\mathbf X$. Fixing a section $x$ of $\mathbf E$ one can recover $\mathbf Z$ as $p^{-1}\{x\}$ where $p_a$ is the projection on $E_a$. 
\end{proof}

%For a Euclidean field, the dimension can be recovered measurably.

\begin{lemma}\label{dim} Let $\mathbf{E}$ be a Euclidean field. The map $a\mapsto \dim(E_a)$ is measurable.
\end{lemma}

\begin{proof} Fix a fundamental family $\{x^n\}$ for $\mathbf{E}$. Thanks to Jung's Inequality \eqref{Jung} the dimension $d$ of $E_a$ can be obtained via the quantity $\sqrt{\frac{d}{2(d+1)}}$ which is the minimal non-negative number $K$ such that 
$\rad\left(\{x^{n_1}_a,\dots,x^{n_k}_a\}\right)\leq K\diam\left(\{x^{n_1}_a,\dots,x^{n_k}_a\}\right)$
where $n_1,\dots,n_k\in \N$.
\end{proof}

Any CAT(0) space $X$ has a visual boundary $\bd X$ (which may be empty). Making this construction pointwise, we can consider, at least in a set-theoretic way, the boundary field of a CAT(0) field $\mathbf X$.

However, the measurable structure is not so clear. We would need a separable metric on each fiber.
One way to endow the boundary of a CAT(0) space with a metric is to consider the angle metric, or the Tits metric which is the length metric associated to the previous one. These are invariant metrics but then $\bd X$ is not separable in general --- for example, the boundary of the hyperbolic plane, endowed with the Tits metric, is an uncountable discrete space.

Actually there is no way to construct a boundary field $\bd \mathbf X$ for a CAT(0) field $\mathbf X$ such that a group  acting on $\mathbf X$ also acts (isometrically) on  $\bd \mathbf X$. Otherwise, any Furstenberg map given by Theorem \ref{main} would be constant because of double metric ergodicity.

One way to avoid these problems,  is to let down the desired invariance of the metric. Recall that  a separable CAT(0) space $X$ embeds continuously (not isometrically) to a subset of the Fr\'echet space $\mathcal{C}(X)$ of continuous functions on $X$ endowed with the distance $d(f,g)=\sum_{n\in \N}2^{-n}\frac{|f(x_n)-g(x_n)|}{1+|f(x_n)-g(x_n)|}$ where $\{x_n\}$ is a dense subset of $X$. This metric topology coincides with the topology of pointwise convergence. More precisely, if $x_0$ is a base point in $X$, the embedding is given by $\iota\colon y\mapsto d(\cdot,y)-d(y,x_0)$. The closure of $\iota(X)$ is a compact metric space which allows us to define a \emph{bordification field} $\mathbf K$ for a CAT(0) field $\mathbf X$ (see \cite[\S 9.2]{MR3044451}).
 Sections of $\mathbf K$ corresponds to some collections of Busemann functions. We define them as follows.

\begin{definition}\label{bndfield}Let $\X$ be a CAT(0) field. We define its \emph{boundary field} $\bd\X$ to be the set of sections $f$ of its bordification field $\mathbf K$ such that for a.e. $a\in A$ there is $\xi_a\in \bd X_a$ with $f_a=\beta_{\xi_a}(\cdot,x_0)$ where $\beta_{\xi_a}$ is the Busemann function associated to $\xi_a\in \bd X_a$.
\end{definition} 
By an abuse of notations we will say that $\xi=(\xi_a)$ is a section of the boundary field. Observe in that case
 for all $x,y$ sections of $\X$, the function 
$a\mapsto\beta_{\xi_a}(x_a,y_a)$
is measurable. Even if we will not need it we observe that one can decide if a section of $\mathbf K$ corresponds to a section of the boundary field.

\begin{lemma} Let $f$ be a section of $\mathbf K$. Let $A'$ be the set of elements $a\in A$ such that there is $\xi_a\in \bd X_a$ with $f_a=\beta_{\xi_a}(\cdot,x_0)$. Then $A'$ is a measurable subset of $A$.
\end{lemma}

\begin{proof}We use the fact that $f_a$ coincides with a Busemann function if it is a limit of points in $\iota(X_a)$ for the topology of uniform convergence on bounded subsets \cite[\S II.8]{MR1744486}. Fix $x^n$ a fundamental family of $\mathbf X$ and $x^n(r)$ fundamental families of the fields $\mathbf B^r$ of closed balls around $x^0$. Now, $a\in A'$ if and only for any  $r>0$ 
$$\inf_{\{n\in\N;\ d(x^n_a,x^0_a)>r\}}\sup_{m\in\N}|f_a(x^m_a(r))-\iota(x^n_a)(x^m_a(r))|=0.$$
\end{proof}

%\begin{remark} Let $\mathbf{E}$ be a measurable field of euclidean spaces with a metric action of some locally compact group $G$. The visual topology on the boundary of each fiber is given by the angle at some point and since each fiber is euclidean this angle does not depend on the choice of point. In the action of $G$ on $\bd \mathbf{E}$ is also metric.
%\end{remark}
The following lemma shows that our notion of boundary field gives something natural in case of a constant field.
\begin{lemma}Let $(X,d)$ be a complete separable metric space and let us denote by $\mathbf{X}$ the measurable field over $(A,\eta)$ with constant fibers equal to $X$. Sections  of $\mathbf{X}$ and measurable maps $A\to X$ are in bijective correspondence. Moreover if $X$ is CAT(0) space then sections of $\bd\mathbf{X}$ and measurable maps $A\to\bd X$ are in bijective correspondence. 

\end{lemma}

\begin{proof}Recall that a fundamental family of $\mathbf{X}$ is given by constant maps $a\mapsto x_n$ where $(x_n)$ is a countable dense subset of $X$. Now a map $f\colon A\to X$ is measurable if and only if for any $x$, $a\mapsto d(x,f(a))$ is measurable if and only if for any $n$, $a\mapsto d(x_n,f(a))$ is measurable.

Now consider the case where $X$ is a CAT(0) space. We use the identification of the boundary of $X$ with the set of Busemann functions vanishing at some base point $x_0$. In this identification, the cone topology corresponds to the topology of uniform convergence on closed balls around $x_0$. In particular, a map $f\colon A\to\bd X$ is measurable if and only if for any $x,y\in X$, $a\mapsto\beta_{f(a)}(x_0,y)$ is measurable  if and only if for any $n,m\in \N$, $a\mapsto\beta_{f(a)}(x_n,x_m)$ is measurable.
\end{proof}

Among CAT(0) spaces, Euclidean spaces have  a special feature: the angle between two points at infinity is the same from any point you look at them. This gives a distance at infinity independent of the choice of a base point and this distance is invariant under the action of the isometry group of the Euclidean space. In case of a field of Euclidean spaces we get the following fact.

\begin{lemma}\label{boueuc}Let $\mathbf{E}$ be a measurable field of Euclidean spaces. The boundary field $\bd\mathbf{E}$ has a structure of a field of CAT(1) spaces for the Tits metric. If $G$ acts isometrically on $\mathbf{E}$ then it acts also isometrically on $\bd\mathbf{E}$.
\end{lemma}

\begin{proof}Each $\bd E_a$ is a complete separable metric space (isometric to a Euclidean sphere) for the Tits metric, and the action of $G$ preserves this distance. What is left to do is to check the measurable structure.

 We choose a fundamental family $(x^n)$ of $\mathbf{E}$ such that there is a section $x^0$ such that for any $n$ and almost all $a$, $d(x^0_a,x^n_a)\neq0$. Now define $\xi^n_a\in\bd E_a$ to be the end point of $[x^0_a,x^n_a)$. We claim that $(\xi^n)$ is a fundamental family for $\bd\mathbf{E}$. We define
\[m(a,n,k)=\min\left\{m;\ d(x^0_a,x^m_a)>k,\ \vert d(x_a^0,x^n_a)+d(x_a^n,x_a^m)-d(x^0_a,x^m_a)\vert<1\right\}\]
and 
\[y^{(n,k)}_a=x_a^{m(a,n,k)}.\]
This way, for any $(n,k)$, $y^{(n,k)}$ is  a section and for almost every $a$, $y^{(n,k)}_a\to\xi^n_a$ as $k\to\infty$.  This shows that for any $n,m$, $a\mapsto\angle(\xi^n_a,\xi^m_a)=\lim_{k\to\infty}\angle_{x_a^0}(y^{(n,k)}_a,y^{(m,k)}_a)$ is measurable. 
\end{proof}

\begin{remark}Let $E$ be a Euclidean space and $\xi,\eta$ be two points at infinity. The very special geometry of $E$ implies the following formula between Busemann functions $\beta_\xi,\beta_\eta$ and the visual angle $\angle_{x_0}(\xi,\eta)$ (which does not depend on $x_0$ and is also the Tits angle) :
\[\sup_{1/2<d(x,x_0)<1}\frac{|\beta_\xi(x,x_0)-\beta_\eta(x,x_0)|}{d(x,x_0)}=2\left(1-\cos(\angle(\xi,\eta)\right).\]
This formula shows that in the case of a measurable field of Euclidean spaces $\mathbf{E}$, the notion of section of $\bd \bf E$ defined in Definition \ref{bndfield} and the notion of section for the structure of a metric field introduced in Lemma \ref{boueuc} coincide.
\end{remark}

Let $E$ be a Euclidean space of dimension $d_0$. It is not hard to define a distance on the set $S$ of  subspheres of dimension $0\leq d<d_0$ in $\bd E$ turning $S$ into a complete separable metric space. The following lemma does the same in a measurable context.

\begin{lemma}\label{subeuc}Let $\mathbf{E}$ be a Euclidean field of constant dimension $d_0$. Let $d$ be a positive integer less than $d_0$. For any $a\in A$, let $S_a$ be the set of subspheres of dimension $d$ of $\bd E_a$. The collection $\mathbf{S}=(S_a)$ has a structure of a metric field. If $G$ acts on $\mathbf{E}$ then it acts also on $\mathbf{S}$ (isometrically).

Let $E^{s}_a$ be the set of Euclidean subspaces $F$ of $E_a$ such that $\bd F=s_a$. Then $\mathbf{E}^s=(E^s_a)$ has a natural structure of Euclidean field such that any section of $\mathbf{E}^s$ corresponds to a Euclidean subfield of $\mathbf{E}$. 

Moreover if $G$ acts on $\mathbf{E}$  and $s$ is an invariant section of $\mathbf{S}$ then  $G$ acts on $\mathbf{E}^s$.

\end{lemma}

\begin{proof}First, we claim that one can construct a fundamental family $(x^n)$ of $\mathbf{E}$ such that for any choice $n_0,\dots,n_d$, and almost all $a$, $x^{n_0}_a,\dots,x^{n_{d}}_a$ is not included in a Euclidean subspace of dimension $<d$ since this condition can be checked only with distances.

For all $a$, we define $S_a$ to be the set of  subspheres of $\bd E_a$ of dimension $d$.  For $s^1,s^2\in S_a$, we define $d(s^1,s^2)$ to be the Hausdorff distance between two compact subspaces associated to the Tits distance on $\bd E_a$.
 Now, for a choice of $N=\{n_0,\dots,n_{d}\}$ we define $s^N_a$ to be the boundary of the affine span of $x_a^{n_0},\dots,x_a^{n_{d_0}}$.
  There are countably many possibilities for $N$ and $\{(s^N)\}_N$ defines a fundamental family of $\mathbf{S}$. If $G$ acts on $\mathbf{E}$ then it acts on $\bd\mathbf{E}$ and since the Hausdorff distance is defined via supremum of some Tits angles, $G$ acts (measurably) on $\mathbf{S}$.

Let $s$ be a section of $\mathbf{S}$ and for any $x^n$ element of the fundamental family of $\mathbf E$ let $F^n_a$ be the unique Euclidean subspace of $E_a$ containing $x^n_a$ such that $\bd F^n_a=s_a$.

The last statement comes from the fact that in that case for any $g\in G$ and all almost $a\in A$, $\alpha(g,a)E^s_a=E^{\alpha(g,a)s_a}_{ga}=E_{ga}^{s_{ga}}$.
\end{proof}

\section{Metric ergodicity and its relative version} \label{sec:metricErg}

In \cite{MR2929597} U. Bader and A. Furman introduced the notion of a boundary pair, which is further developed in their yet unpublished paper \cite{Bader:2025}. In this section we review this theory.

We fix a locally compact second countable group $G$ for the rest of this section. All conditions of measurability in $G$ will be relative to the Haar measure class.

\begin{definition}[{\cite[Definition 4.1]{Bader:2013zr}}] Let $(A,\eta)$ be a $G$-space. The action $G\action (A,\eta)$ is \emph{metrically ergodic} if for any action of $G$ by isometries on a complete separable metric space $(X,d)$, any $G$-equivariant measurable map from $A\to X$ is essentially constant.

If the diagonal action $G\action A\times A$ is metrically ergodic, we say that $G\action A$ is \emph{doubly metrically ergodic}.
\end{definition}

\begin{remark} We will only use complete separable metric spaces. However, this is not an important restriction.  About completeness, one may consider the extended action on the completion $\overline X$ and observe that $\overline X\setminus X$ has zero measure. Moreover one may reduce to separable spaces as the following argument shows. % as one may observe that if $G\action A$ is metrically ergodic then any measurable $G$-map to a metric space on which $G$ acts by isometries is also essentially constant.

Let $f$ be a  measurable map $A\to X$. Thanks to \cite{MR612620}, $f$ is actually $\eta$-measurable in N~. Bourbaki's sense and thus separably valued (see \cite[IV \S 5 No 5, Theorem 4]{MR2018901}), meaning that  there is $X'\subseteq X$ closed and separable such that $f_*\eta(X\setminus X')=0$.
 Now, using separability of $X'$, it can be easily checked that the set $S=\{x\in X \mid \forall \varepsilon>0,\ f_*\eta(B(x,\varepsilon))>0\}$  is a closed subset of $X'$ which is $G$-invariant and satisfies $f_*\eta(X\setminus S)=0$.
\end{remark}

Below we present a \emph{relative} notion of metric ergodicity as well \cite{MR2929597}. The definition that we give here is not exactly the one given in the paper \cite{MR2929597}, but a version of it modified in order to fit in the context of measurable fields of complete separable metric spaces

% We refer to Appendix \ref{measurable} for details and more references about measurable fields of complete separable metric spaces.

\begin{definition}
Let $(A,\eta)$ and $(B,\nu)$ be two Lebesgue spaces. A measurable map $\pi\colon A\to B$ is a \emph{factor map} if $\pi_*\eta$ and $\nu$ are in the same measure class. If  $A$ and $B$ are $G$-spaces and $\pi$ is $G$-equivariant then we say that $\pi$ is a $G$-\emph{factor}.
\end{definition}

\begin{definition}[{\cite{Bader:2025}}]Let $(A,\eta)$ and $(B,\nu)$ be two Lebesgue spaces and $\pi\colon A\to B$ be a factor map. Let $\mathbf{X}$ be a metric field over $(B,\nu)$. A \emph{relative section} is a map $\varphi\colon A\to \sqcup_{b\in B} X_b$ such that :
 \begin{itemize}
\item for  all $a\in A$, $\varphi(a)\in X_{\pi(a)}$,
\item for any section $x$ of $\mathbf{X}$, $a\mapsto d(x(\pi(a)),\varphi(a))$ is measurable.
\end{itemize}
If $\pi$ is a $G$-factor and $G$ acts on $\mathbf{X}$ via a cocycle $\alpha$, such a relative section is said to be \emph{invariant} if for almost every $a$ and all $g\in G$, $\varphi(ga)=\alpha(g,\pi(a))\varphi(a)$. 
\end{definition}

\begin{definition}[{\cite{Bader:2025}}]\label{simsit}We say that the $G$-map $\pi\colon A\to B$ is \emph{relatively metrically ergodic} (or equivalently $G\action A$ is \emph{metrically ergodic relatively} to $\pi$) if any invariant relative section coincides with a section. In other words, for any $G$-metric field $\mathbf{X}$ and any invariant relative section $\varphi$, there is an invariant section $x$ of $\mathbf{X}$ such that for almost all $a\in A$, $\varphi(a)=x(\pi(a))$.
\end{definition}

The following lemma shows how relative metric ergodicity implies metric ergodicity. Actually metric ergodicity of $G\action A$ is equivalent to metric ergodicity of $G\action A$ relatively to the projection to a point. 

\begin{lemma}[{\cite{Bader:2025}}]Let $(A,\eta),$ $(B,\mu)$ be two $G$-spaces. If $A\times B\to B$ is relatively metrically ergodic then $G\action A$ is metrically ergodic.
\end{lemma} 

\begin{proof} Let $\phi\colon A\to X$ be an equivariant $G$-map to some complete separable metric space. Consider $\mathbf{X}$ be the trivial field $X\times B$ over $B$ and define $\varphi(a,b)=\phi(a)$. The map $\varphi$ is an invariant relative section and thus does not depend on $a$, that is  $\phi$ is essentially constant.
\end{proof}

\begin{definition}[{\cite{Bader:2025}}]Let $(B_-,\nu_-),(B_+,\nu_+)$ be $G$-spaces, we say that $(B_-,B_+)$ is a $G$-\emph{boundary pair} if 
\begin{itemize}
\item the actions  $G\action B_+$ and $G\action B_-$ are amenable in Zimmer's sense \cite{MR776417},
\item both first and second projections $B_-\times B_+\to B_\pm$ are relatively ergodic.
\end{itemize}
A $G$-space $(B,\mu)$ is $G$-\emph{boundary} if $(B,B)$ is a $G$-boundary pair.
\end{definition}

Let $\mu$ be a probability measure on $G$. Recall that a $(G,\mu)$-\emph{space} is a $G$-space $(A,\eta)$ such that $\mu\ast\eta=\eta$, where the convolution measure $\mu\ast\eta$ is the pushforward measure of $\mu\times\eta$ under the map $(g,a)\mapsto ga$. Such a measure 
$\nu$ is called $\mu$-harmonic or $\mu$-stationary in the literature. Let $i\colon G\to G$ be the inversion given by $i(g)=g^{-1}$ for any $g\in G$. We denote by $\check{\mu}$ the probability measure $i_*\mu$. Recall that $\mu$ is \emph{symmetric} if $\check{\mu}=\mu$. \\

For the remainder of this section, we fix a spread out non-degenerate probability measure $\mu$ on $G$, that is $\mu$ is absolutely continuous with respect to a Haar measure and its support generates $G$ as semigroup. Let  $(B,\nu)$ be the Poisson boundary associated to $\mu$.  We also denote by $(\check{B},\check{\nu})$ the Poisson boundary of $(G,\check{\mu})$. We refer to \cite{MR0146298,MR2006560} for notions and references about Poisson boundaries and related ergodicity properties. We emphasize that $(B,\nu)$ and $(\check{B},\check{\nu})$ are respectively a $(G,\mu)$-space and a $(G,\check{\mu})$-space.

\begin{theorem}[{\cite{Bader:2025}}]\label{boundary}The pair $(\check B,B)$ is a $G$-boundary pair.
\end{theorem}

This  will be deduced from the following statement.

\begin{theorem}[{\cite{Bader:2025}}]\label{rightGspace}Let $(A,\eta)$ be a $(G,\check \mu)$-space. The factor map $A\times B\to A$ is relatively ergodic.
\end{theorem}

\begin{corollary}[{\cite{Bader:2025}}]\label{doubleergodicity}
The diagonal action of $G$ on $ (\check B\times  B,\check\nu\times \nu)$ is metrically ergodic.
\end{corollary}

\begin{remark}\label{dme} The same argument as in Corollary \ref{doubleergodicity} shows that if $(B_-,B_+)$ is a $G$-boundary pair then $G\action B_-\times B_+$ is metrically ergodic. In particular if $B$ is a $G$ boundary then $G\action B$ is doubly metrically ergodic.\end{remark}

\begin{proof}[Proof of Corollary \ref{doubleergodicity}]
Let $U$ be a metric separable space on which $G$ acts continuously and by isometries. Assume $f:\check B\times  B\to U$ is a $G$-equivariant measurable map. 
%Note that the map $\tilde f:\check B\times B\to U$ defined by $\tilde f(b,b')=f(b',b)$ is also measurable and $G$-equivariant.
Take $A=\check B$. It follows from the Theorem \ref{rightGspace} that $f$ only depends on the first coordinate. 

Now consider the measure $\check \mu$. Then $B$ is a  $(G,\displaystyle\check{\check {\mu}})$-space (in other words,  a $(G,\mu)$-space), so we can apply Theorem \ref{rightGspace} to $B\times \check B$.
 This implies that $f$ does not depend on the first coordinate.

Putting together the two results, we see that $f$ does not depend on the first, nor on the second coordinate. Hence, $f$ is constant.

\end{proof}

The following proposition is a key tool in the proof of Theorem \ref{boundary}. It combines Poincar\'e recurrence theorem for $A$ and SAT property for $B$. Recall that SAT, which means \emph{strongly almost transitive}, is  a weak mixing property introduced by W. Jarowski \cite{MR1285567}.

 \begin{theorem}[{\cite{Bader:2025}}]\label{thm:SATrec}
Let $(A,\eta)$ be a $(G,\check\mu)$-space and $Y\subset A\times B$ be a set of positive $\eta\times\nu$-measure. For $a\in A$, denote by $Y_a$ the set $\{b\in B\mid (a,b)\in Y\}$. Then, for almost every $a\in \pi_1(Y)$, and for every $\varepsilon>0$, there is a $g\in G$ such that 
 \begin{enumerate}[label=(\roman*)]
  \item $\nu(gY_a)>1-\varepsilon$
  \item $ga\in\pi_1(Y)$ where $\pi_1$ is the first projection $A\times B\to A$.
 \end{enumerate}
 \end{theorem}

 \begin{proof}
 We will use the definition of the Poisson boundary as a space of ergodic components of the space of increments of the random walk $\Omega=G\times G^\bN$ by the shift $S$.
 
 We can define another shift $T$ from $A\times \Omega$ to itself, defined by $T(a,g,\omega_1,\omega_2,\dots)=(g^{-1}a,\omega_1,\omega_2,\dots)$. Since $A$ is a $(G,\check\mu)$-space, we have $\check\mu\ast\eta=\eta$, and we conclude that $T$ preserves the measure $m:=\eta\times\mu\times\mu^{ \bN}$ on $A\times\Omega$.
 
 Let us consider the fiber product $X=A\times_G\Omega$. This is defined as the quotient space of $Y\times \Omega$ by the relation $(a,hg,\omega_1,\dots)\sim (h^{-1}a,g,\omega_1,\dots)$, for all $h\in G$. It follows that the space $X$ is isomorphic to $A\times G^\bN$. The pushforward of the measure on $A\times \Omega$ to $X$ is simply the measure $\eta\times\mu^{\bN}$. Furthermore, the shift $T$ preserves the equivalence relation, so it still acts on $X$, and  also preserves the measure.

 Let $Y\subset A\times B$ be such that $\eta\times\nu(Y)>0$. We can consider the preimage of $Y$ in $A\times \Omega$, and then push it forward to get a subset $\tilde Y$ of $X$. First we note that, by Poincar\'e recurrence theorem, we have that, for almost every $x\in\tilde Y$, there are infinitely many $n\in\bN$ such that  $T^n x\in\tilde Y$.
 
 If $a\in A$, define the set $\tilde Y_a=\{\omega\in G^\bN\mid (a,\omega)\in\tilde Y \}$. By Fubini, we have  $\mu^\N(\tilde Y_a)>0$ for almost every $a\in\pi_1(\tilde Y)$ (where $\pi_1:A\times G^\bN\to A$ denotes the first projection as well). In other words, for almost every $x\in\tilde Y$, we have $\nu(Y_{\pi_1(x)})>0$.
 
 Let us fix $x\in\tilde Y$ in the intersection of the two conull sets defined above. Let $a=\pi_1(x)$.
  The set $Y_a$ is of positive measure, so its characteristic function $\chi_{Y_a}$ is an element of $L^\infty(B)$ which is not zero. 
  Let $h$ be its Poisson transform. Recall that it is defined as $$h(g)=\int_B\chi_{Y_a}\dd g_*\nu=\nu(g^{-1} Y_a).$$
   This function is a non-zero bounded harmonic function on $G$. 
 
 By definition of the Poisson transform, we have, for almost every $\omega=(\omega_1,\omega_2,\dots)\in \tilde Y_a$, 
 \begin{align*}
 \nu((\omega_1\omega_2\dots\omega_n)^{-1} \tilde Y_a)&= h(\omega_1\omega_2\dots\omega_n)\\
 &\underset{n\to+\infty}\longrightarrow \chi_{Y_a}(\omega)=1.
 \end{align*}
 
 In particular, we might pick $n$ large enough so that $\nu((\omega_1\omega_2\dots\omega_n)^{-1} \tilde Y_a)>1-\varepsilon$ and also satisfying $T^n(a,\omega)\in\tilde Y$. Setting $g=(\omega_1\omega_2\cdots\omega_n)^{-1}$, we have by definition $\nu(g Y_a)>1-\varepsilon$. Furthermore since $T^n(a,\omega)\in Y$, we have $(ga,\omega_{n+1},\dots)\in Y$ ; hence $Y_{ga}\neq\emptyset$.

 \end{proof}

\begin{proof}[Proof of Theorem \ref{rightGspace}]Let $\mathbf{X}$ be a metric field over $(A,\eta)$ on which $G$ acts via the cocycle $\alpha$. Observe that we may assume that any fiber $X_a$ has diameter at most 1. Otherwise, we replace $d_a$ by $\max(d_a,1)$ and we obtain a new  metric field over $B$ on which $G$ acts as well.\\
Let $\varphi$ be an invariant relative section. Let us define
$$f(a)=\int_{B\times B}d_a(\varphi(a,b),\varphi(a,b'))\dd \nu(b)\dd\nu (b').$$

Our assumption implies that $f$ is not essentially $0$. In particular, there is an $r>0$ such that $A(r):=f^{-1}([r,+\infty))$ is of positive measure.

Take a small $\delta>0$. Let $\{x^n\}$ be a fundamental family of $\mathbf{X}$. Then there is $n\in\N$ such that $Y=\{(a,b)\in Y_r\times B\mid d(\varphi(a,b),x^n_a)\leq\delta \}$ has positive measure, say $>\varepsilon$.

By  Theorem \ref{thm:SATrec}, this implies that there is a $a\in \pi_1(Y)$ and a $g\in G$ such that $ga\in\pi_1(Y)$ and $\nu(gY_a)>1-\varepsilon$.

Now for $(ga,b)\in gY_a$ and $(ga,b')\in gY_a$ (in other words, $b,b'\in Y_{ga}$), we know that  $ d(\varphi(ga,b),\alpha(g,a)x^n_a)\leq\delta$ and $ d(\varphi(ga,b'),\alpha(g,a)x^n_a)\leq\delta$. So these two points, $\varphi(ga,b)$ and
$\varphi(ga,b')$ are in the same ball of radius $\delta$. Therefore we have $d(\varphi(ga,b),\varphi(ga,b'))\leq2\delta$.

Let us decompose $B\times B$ as $$(gY_a\times gY_a)\cup ((B\setminus gY_a) \times gY_a) \cup( gY_a\times (B\setminus gY_a))\cup ((B\setminus gY_a \times B\setminus gY_a)).$$
We know that $\nu(B\setminus gY_a)<\varepsilon$. Hence, $\nu((B\times B)\setminus (gY_a\times gY_a))<\varepsilon^2+2\varepsilon$.

Therefore, we have 
$$f(ga)<\int_{gY_a\times gY_a} d(\varphi(ga,b),\varphi(ga,b'))\dd\nu(b)\dd\nu(b')+2\varepsilon+\varepsilon^2$$

It follows that $f(ga)<2\delta+2\varepsilon+\varepsilon^2$. Since we also assumed that $ga\in\pi_1(A)$, we have that $g\in A(r)$. Therefore, $f(ga)>r$. Since our choice of $\delta$ and $\varepsilon$ is arbitrary, we get a contradiction.

\end{proof}

\begin{proof}[Proof of Theorem \ref{boundary}]The amenability of $G\action B$ is due to Zimmer \cite{MR0473096} and relative ergodicity for $\check B\times B\to \check B$ comes from Theorem \ref{rightGspace}. Relative ergodicity for $\check B\times B\to  B$ is similar using that $\check{\check\mu}=\mu$.
\end{proof}

\section{Furstenberg maps}\label{Bmap}

The geometric part of Theorem \ref{main} uses the following version of Adams-Ballmann theorem \cite{MR1645958,MR2558883} as a key tool.

\begin{theorem}[Equivariant Adams-Ballmann theorem {\cite[Theorem 1.8]{MR3044451}}]\label{AB} Let $(A,\eta)$ be an ergodic $G$-space such that $G\action A$ is amenable. Let $\X$ be a CAT(0) field of finite telescopic dimension.

If $G$ acts on $\X$ then there is an invariant section of the boundary field $\bd\X$ or there exists an invariant Euclidean subfield of $\X$.
\end{theorem}

Before beginning the proof of Theorem \ref{main} we can give a hint about its structure. Thanks to Theorem \ref{AB} used for a constant field, the only case where there is something to prove is the case where there is an invariant Euclidean subfield and no Furstenberg map. In that case, we choosing a minimal such invariant Euclidean subfield. Then, we analyze the possible relative positions of two subflats of $X$ by using Proposition \ref{relpos}. Using relative metric ergodicity we conclude that this  subfield must actually be constant equal to some fixed Euclidean subspace.

Theorem \ref{main} will actually be deduced straightforwardly from this more general theorem for boundary pairs.

\begin{theorem}\label{mainpairs} Let $X$ be a CAT(0) space of finite telescopic dimension and let $G$ be a locally compact second countable group acting continuously by isometries on $X$ without invariant flats. If $(B_-,\nu_-)$ and  $(B_+,\nu_+)$ form a  $G$-boundary pair then there exist  measurable $G$-maps $\varphi_\pm \colon B_\pm\to \bd X$.
\end{theorem}

\begin{proof} First we do some reductions to prove the theorem. Clearly, it suffices to prove the theorem in case there is no fixed point of $G$ in $\bd X$ since otherwise we have a trivial map $B_\pm\to\bd X$. Let us assume there is no fixed point at infinity. In that case, there is a $G$-invariant closed convex space with a minimal $G$-action. That is there is no non-trivial $G$-invariant closed convex subset  \cite[Proposition 1.8 (ii)]{MR2558883}. Thus we may assume that this space is actually $X$ itself and $X$ is the closed convex hull of any orbit. Orbits are separable since $G$ acts continuously and $G$ is second countable. Thus $X$ is separable.

Let $X=H\times Y$ be the de Rham decomposition of $X$ where $H$ is a Hilbert space --- of finite dimension since $X$ has finite telescopic dimension --- and $Y$ is another CAT(0) space of finite telescopic dimension. The action of $G$ on $X=H\times Y$ is diagonal and since the action of $G$ on $X$ is minimal, both actions $G\action H$ and $G\action Y$ are minimal. Observe it suffices to prove the result for $G\action Y$ because $\bd Y\subseteq\bd X$ and if $F$ is a flat of $Y$ then $H\times F$ is a flat of $X$. We may assume that $X=Y$, that is $X$ has no Euclidean de Rham factor of positive dimension.

Thus we reduce the proof the following case: $X$ is separable, has trivial  Euclidean de Rham factor and the action of $G$ is minimal.\\

We consider the constant fields $\mathbf{X}^\pm$ over $B_\pm$ with fibers $X$. These fields are endowed naturally with actions of $G$ and we can apply Theorem \ref{AB}. For each $\mathbf X^\pm$ we get a map from $B_\pm$ to $\bd X$ or an invariant Euclidean subfield of $\mathbf X^\pm$. 
In the first case, one gets two $G$-maps  $B_\pm\to \bd X$, we are done. Up to permuting $B_+$ and $B_-$, one may assume either that we get two invariant Euclidean subfields $\mathbf E^-$ and $\mathbf E^+$ of respectively $\mathbf X^-$ and $\mathbf X^+$ and there is no $G$-map $B_\pm\to \bd X$ (\textbf{Case (I)}) or we get a $G$-map $b\mapsto \xi_b$ from $B_-$ to $\bd X$ and a Euclidean subfield $\mathbf E$ of $\mathbf X^+$ (\textbf{Case (II)}).

Readers interested only in the case of a $G$-boundary $B$, that is $B=B_-=B_+$, may read only case (I) with $\mathbf X^+=\mathbf X^-$ and $\mathbf E^+=\mathbf E^-$.\\

\textbf{Case (I)}: Thanks to Lemma \ref{dim}, the map $b\mapsto \dim (E^-_b)$ is measurable and $G$-invariant. Thanks to ergodicity of $G\action B_-$, this dimension is essentially constant. We may assume that this dimension is minimal among dimensions of invariant Euclidean subfields of $\mathbf X_-$. For simplicity, we note $E_b=E^-_b$ and $E_{b'}=E^+_{b'}$, for $b\in B_-$ and $b'\in B_+$.

We aim to apply Proposition \ref{relpos} to each pair $(E_b,E_{b'})$ for $(b,b')\in B_-\times B_+$. First, we show that the four conditions in the proposition are measurable and $G$-invariant; thanks to ergodicity of $G\action B_-\times B_+$, exactly one condition will be satisfied for almost every $(b,b')\in B_-\times B_+$.

Let $\{x^n_b\}_{b\in B_-}$ be a fundamental family of $\mathbf{E^-}$, and $\{x^n_{b'}\}_{b'\in B_+}$ be a fundamental family of $\mathbf E^+$. We define $d(b,b')=\inf_{n,m}d(x^n_b,x^m_{b'})$. The function $d:B_-\times B_+\to \R$ is measurable. The infimum distance between $E_b$ and $E_{b'}$ is not achieved if and only if for any $N\in\N$ and any $n$ such that $d(x^0_b,x^n_{b'})<N$, one has $\inf_{m}d(x^n_b,x^m_{b'})>d(b,b')$. This condition is measurable and $G$-invariant.

If the minimal distance between $E_b$ and $E_{b'}$ is achieved, then the subset of $E_b$ where the distance is achieved is not bounded if and only if for any $N\in\N$, there exists $n_1,n_2$ such that $d\left(x^{n_1}_b,x^{n_2}_b\right)>N$ and $\inf_{m}d\left(x^{n_i}_b,x^m_{b'}\right)\leq1+d(b,b')$. 

Now, we know that exactly one of these four relative positions given in Proposition \ref{relpos} between $E_b$ and $E_{b'}$ happens for almost every $(b,b')$. We treat the four cases independently.\\

{\it Case (i):} If the infimum distance is not achieved, choose a section $y$ of $\mathbf{E^-}$ and consider (for $b'$ in a set of full measure) the nonincreasing sequence of subfields $\mathbf{X}^n$ where we define $X_b^n$ as $X_b^n=\left\{x\in E_b,\ d(x,E_{b'})\leq d(b,b')+\frac{1}{n}\right\}.$
We can now apply \cite[Proposition 8.10]{MR3044451} and obtain a section $\xi_b(b')$ of  the boundary field $\bd\mathbf{E}^-$. This field  is a metric field with a $G$-action by Lemma \ref{boueuc}. 
By construction, $\xi$ satisfies $\alpha(g,b)\xi_b(b')=\xi_{gb}(gb')$. This means that $\xi:B_-\times B_+\to \sqcup \bd E_b$ is an invariant section relative to the first projection $B_-\times B_+\to B_-$. 
 Now, thanks to relative metric ergodicity, $\xi$ does not depends on $b'$. So we can interpret $b\mapsto\xi_b$ as a measurable $G$-map from $B_-$ to $\bd X$. This yields a contradiction.\\

If the minimal distance is achieved, we define  $Y_b(b')=\{x\in X_b,\ d(x,E_{b'})=d(b,b')\}$.
Then the family $\mathbf{Y}=(Y_b(b'))_b$ 
is a subfield of $\mathbf{X}$ (which depends on $b'$).\\
 
{\it Case (ii):} Assume the minimal set $Y_b(b')$ is bounded. Let $x_b(b')$ be its circumcenter. Then for almost all $b'$, the family of circumcenters $\{x_b(b')\}$ is a section of $\mathbf E^-$ \cite[Lemma 8.7]{MR3044451} which satisfies $\alpha(g,b)x_b(b')=x_{gb}(gb')$, that is an invariant relative section. Relative metric ergodicity shows that this map is essentially constant. So its essential image is a point which is fixed by $G$, which is excluded.\\

In the other case, thanks to \cite[Proposition 9.2]{MR3044451}, we can write $Y_b(b')=F_b(b')\times T_b(b')$ where $\mathbf F=(F_b(b'))_b$ is a maximal Euclidean subfield of $\mathbf Y$ and $(T_b(b'))_b$ is a subfield of $\mathbf Y$; both subfields depending on $b'$.\\

Before attacking Case (iii), which is harder, let us treat Case (iv).\\

{\it Case (iv):} If for almost $(b,b')$, $T_b(b')$ is not bounded then its boundary is not empty and its circumradius $<\pi/2$ thanks to Proposition \ref{relpos}. Let $\xi_b(b')$ be the circumradius of $\bd T_b(b')$. Thanks to Lemma \ref{cenCAT1}, we get a measurable $G$-map from $B_-\times B_+$ to the  metric field $\bd\mathbf E^-$ over $B_-$. Once again relative metric ergodicity implies that $\xi_b(b')$ coincides with  a $G$-map $b\mapsto\xi_b$ and we are done.\\

{\it Case (iii):} If $T_b(b')$ is bounded for almost every $(b,b')$ then we set $t_b(b')$ to be its circumcenter and $E'_b(b')=F_b(b')\times\{t_b(b')\}$.

 The dimension of $E'_b(b')$ is measurable (Lemma \ref{dim}) in $(b,b')$ and $G$-invariant. Thus, this dimension is essentially constant equal to some $k\in\mathbb N$.

Now $(b,b')\mapsto \bd E'_b(b')$ is a $G$-section of the metric field $\mathbf{S}$ of Euclidean subspheres of dimension $k-1$  associated to $\mathbf{E}^-$ (see Lemma \ref{subeuc}). Relative metric ergodicity implies that $\bd E'_b(b')$ is essentially equal to some $s_b$ and $s=(s_b)$ is a $G$-invariant section of $\mathbf{S}$. Using the second part of Lemma \ref{subeuc}, let us consider the field $\mathbf{E}^s$ such that for any $b\in B$, $E^{s}_b$ is the set of Euclidean subspaces of $E_b$ with boundary $s_b$. By definition $(b,b')\mapsto E'_b(b')$ is a relative invariant section of $\mathbf{E}^s$ and relative metric ergodicity implies that there is a section of $\mathbf{E}^s$, that is a Euclidean subfield $\mathbf{E}'$ of $\mathbf{E}^-$ such that for almost every $(b,b')$, $E'_b(b')=E'_b$.

Our assumption on the minimality of the dimension of $\mathbf{E}^-$ implies that $E'_b=E_b$ for almost every $b\in B_-$.

%Indeed, suppose not.
%If $\mathbf E'$ is a strict subfield of $\mathbf E$, we can run again the same argument, replacing $\mathbf E$ by $\mathbf E'$. In particular the dimension of $E'(b)$ is strictly less than the dimension of $E(b)$. Then if Case (i), (ii) or (iv) happen for $\mathbf E'$, we directly have the result. 
%If not, then this means that we can again replace $E'(b)$ by a Euclidean subfield of stricly smaller dimension. Since we can decrease the dimension only finitely many times, there will be a point at which we have indeed $E'(b)=E(b)$.

 Going back to the definition of $E'_b=E_b(b')$, we see that the fact that $E'_b=E_b$ implies that the distance function $d(\cdot,E_{b'})$ must be constant in restriction to $E_b$. That is to say that $E_b$ and $E_{b'}$ are parallel for almost all $b,b'\in B_-\times B_+$.
 
Fubini's theorem tells us that there is $b\in B_-$ such that for a set $B'\subseteq B_+$ of full measure and for any $b'\in B'$, $E_{b'}$ is parallel to $E_b$. Fix $\Gamma\leq G$ a dense countable subgroup of $G$ and set $B_0=\cap_{\gamma\in \Gamma}\gamma B'$ which is a $\Gamma$-invariant subset of full measure in $B_+$. Lemma  \ref{flatproduct}  implies that the closed convex hull of the union of the $\{E_b\}_{b\in B_0}$ splits as some product $E\times T$ where each $E_b$ is parallel to $E$. Observe that $\Gamma$ preserves this convex set. By continuity of the action, the group $G$ preserves this set as well and by minimality this set is $X$. Now the assumption that $X$ has trivial Euclidean de Rham factor implies that $E$ is a point and thus the dimension of $E_{b}$ for almost every $b$ is zero. Finally $E_{b},E_{b'}$ are points of $X$ and  double metric ergodicity (Remark \ref{dme}) implies there is a fixed point.\\

\textbf{Case (II)}: Fix a point $x\in X$ and for $b\in B_-$ let $\beta_b$ be the Busemann function associated to $\xi_{b}$ vanishing at $x$. Now for $(b,b')\in B_-\times B_+$, look at the restriction $f_{b,b'}$ of $\beta_{b}$ to $E_{b'}$. It is a convex function and arguing as above, if it is not constant, one gets a relative section of $\mathbf E$ or of $\bd \mathbf  E$. Once again relative ergodicity gives an invariant Euclidean subspace of $X$ or a fixed point at infinity.

If $f_{b,b'}$ is constant, the situation is different. Actually since this situation is a measurable $G$-invariant condition on $(b,b')$ thanks to double ergodicity it holds for almost all $(b,b')\in B_-\times B_+$. Applying Lemma \ref{parallel field} to $\mathbf E$, we get that the subfield $\mathbf Y$ of flats parallel to $\mathbf E$ splits as $\mathbf E\times \mathbf Z$ and $G$ acts on $\mathbf Z$ with an invariant section $z$. Lemma \ref{flatbusemann} shows that $\xi$ is actually a section of $ \bd \mathbf Z$.

Let $\mathbf B^r$ be the subfield of $\mathbf Z$ consisting of closed balls of radius $r>0$ around $z$. Thanks to Lemma \ref{field balls}, $G$ acts on this CAT(0) field. Let $z_{b'}^r(b)$ be the point on the ray from $z_{b'}$ to $\xi_b$ at distance $r$ from $z_{b'}$. This is an invariant relative section of $\mathbf B^r$ and thanks to relative ergodicity, there is an invariant section $z^r_{b'}$ of $\mathbf B^r$ such that $z_{b'}^r(b)=z^r_{b'}$ for almost $(b,b')\in B_-\times B_+$.

Let us define $\eta_{b'}$ to be $\lim_{r\to\infty}z^r_{b'}$. By construction this is a invariant section of $\bd \mathbf X_+$ that is a $G$-map $\varphi_+\colon B_+\to\bd X$ as desired.
\end{proof}

\appendix

\section{Non-emptyness of the boundary of minimal CAT(0) spaces of finite telescopic dimension}\label{nonemptyness}

P. Py pointed out that the proof of Proposition 1.8 in \cite{MR2558883} uses implicitly the fact that the boundary of an unbounded CAT(0) space of finite telescopic dimension is non empty as soon as its group of isometries acts minimally (meaning that is there is no non-trivial  closed convex subset which is invariant under all isometries). The proof of \cite[Proposition 1.8]{MR2558883} was completed in \cite{erratumCL12} but the question about the non-emptyness of the boundary still remained. 

We answer positively this question. The example of a rooted tree with an edge of length $n$ attached to the root for any $n\in\N$ shows that the hypothesis of no fixed point for all isometries is essential. Observe that if the boundary is empty then any isometry is elliptic since any hyperbolic or parabolic isometry has a fixed point at infinity. Moreover, \cite[Theorem 1.6]{MR2558883} shows that the isometry group of such a minimal space cannot be amenable.

We emphasize that Theorem \ref{nem} relies on Theorem \ref{AB} which only uses \cite[Proposition 1.8(ii)]{MR2558883} where there is no implicit assumption. So, the gap in the cited paper is filled in. 

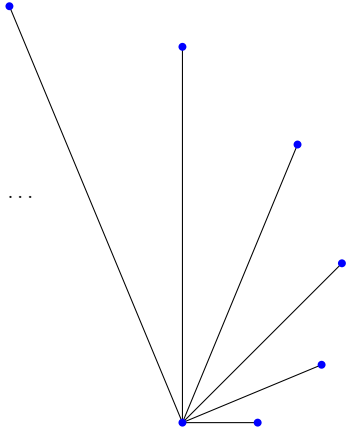
\begin{figure}[h]
\begin{center}
\definecolor{qqqqff}{rgb}{0,0,1}
\begin{tikzpicture}[%line cap=round,line join=round,>=triangle 45,
x=1.0cm,y=1.0cm]
\clip(-2.93,-1.05) rectangle (3.05,6.03);
\draw (0,0)--(1,0);
\draw (0,0)-- (1.85,0.77);
\draw (0,0)-- (2.12,2.12);
\draw (0,0)-- (1.53,3.7);
\draw (0,0)--(0,5);
\draw (-2.3,5.54)-- (0,0);
\begin{scriptsize}
\fill [color=qqqqff] (0,0) circle (1.5pt);
\fill [color=qqqqff] (1,0) circle (1.5pt);
\fill [color=qqqqff] (1.85,0.77) circle (1.5pt);
\fill [color=qqqqff] (2.12,2.12) circle (1.5pt);
\fill [color=qqqqff] (1.53,3.7) circle (1.5pt);
\fill [color=qqqqff] (0,5) circle (1.5pt);
\fill [color=qqqqff] (-2.3,5.54) circle (1.5pt);
\draw(-2.13,3) node {$\dots$};
\end{scriptsize}
\end{tikzpicture}
\caption{An unbounded CAT(0) space of dimension 1 with empty boundary} 
\end{center}
\end{figure}

\begin{proof}[Proof of Theorem \ref{nem}] First of all, we prove the result for finitely generated groups. Let $G$ be a finitely generated group acting on some CAT(0) space of finite telescopic dimension $X$ with empty boundary. This last assumption implies that there is a closed convex subset on which $G$ acts minimally. We may assume this subset is $X$ itself. Since $G$ is countable then the closed convex hull of any orbit is $G$-invariant and separable. The minimality assumption implies $X$ coincides with any such closed convex hull and thus $X$ is separable. Since $G$ is finitely generated then $G$ has a symmetric probability measure $\mu$ uniformly supported on a symmetric finite set of generators and thus the Poisson boundary $B$ associated to $\mu$ is a $G$-boundary.

We can apply Theorem \ref{AB} for the constant field of CAT(0) spaces $X$ over $B$ and we obtain a measurable map $B\to\bd X$ or an invariant subfield of flats. In any case, we get an non empty boundary or an equivariant map $B\to X$. In this last case double metric ergodicity for $G\action B$ implies that this map is essentially constant and we get a $G$-fixed point.\\

We conclude in the general case in the following way: assume $X$ has empty boundary and let $G$ be the group of all isometries of  $X$. Since $\bd X=\emptyset$, $X$ is $\mathscr{T}_c$-compact (see \S2.1). Let $\mathscr{F}$ be the collection of all finite subsets of $G$. For $F\in\mathscr{F}$ let $G_F$ be the subgroup generated by $F$. We just proved that $G_F$ has a non-empty, closed and convex set of fixed points that we denote by $X_F$. The collection $\{X_F\}_{F\in\mathscr{F}}$ is then a filtering family of $\mathscr{T}_c$-closed non-empty subspaces, by compactness $\cap_{F\in\mathscr{F}}X_F\neq\emptyset$ and thus one gets a $G$-fixed point.
\end{proof}

\begin{proof}[Proof of Corollary \ref{bm}] The proof goes as in \cite[Proposition 1.5]{MR2574740}. We reproduce it for reader's convenience. Theorem \ref{nem} shows that  $\bd X$ is not empty. Assume $\bd X$ has radius at most $\pi/2$ then there is a canonical point $\xi\in\bd X$ fixed by all isometries and such that the closed ball of radius $\pi/2$ --- for the Tits distance --- around $\xi$ coincides with $\bd X$. If $g\in\Isom(X)$ does not let the Busemann function $\beta_\xi$ invariant, then $g$ is an hyperbolic isometry and translates along some geodesic line with extremities $\eta_-,\eta_+\in\bd X$. One of this point does not belong to the closed ball of radius $\pi/2$ around $\xi$ yielding a contradiction.

Now, if $\bd X$ has radius larger than $\pi/2$, there is a minimal closed convex subset $Y\subseteq X$ with $\bd Y=\bd X$. The union $X_0$ of all such minimal spaces is closed convex and can be decomposed as $X_0=Y\times Z$. Moreover $X_0$ is $\Isom(X)$-invariant with a diagonal action. Minimality implies that $X=X_0$. If $\bd Z\neq\emptyset$ one has $\bd Y=\bd X=\bd Y*\bd Z$ leading to a contradiction. Thus $\bd Z=\emptyset$ and Theorem \ref{nem} implies $\Isom(X)\action Z$ has a fixed point. This fixed point is the whole of $Z$ by minimality and thus $X=Y$.
\end{proof}

\bibliographystyle{halpha}
\bibliography{biblio.bib}
\end{document}